\documentclass{amsart}
\usepackage{amssymb}
\usepackage{amsfonts}
\usepackage{amssymb}
\usepackage{amsmath}
\usepackage{amsthm}
\usepackage{enumerate}
\usepackage{tabularx}
\usepackage{centernot}
\usepackage{mathtools}
\usepackage{amsthm,amssymb}
\usepackage{etoolbox}

\renewcommand{\leq}{\leqslant}
\renewcommand{\geq}{\geqslant}

\makeatletter
\def\subsection{\@startsection{subsection}{3}%
  \z@{.5\linespacing\@plus.7\linespacing}{.3\linespacing}%
  {\bfseries\centering}}
\makeatother

\makeatletter
\def\subsubsection{\@startsection{subsubsection}{3}%
  \z@{.5\linespacing\@plus.7\linespacing}{.3\linespacing}%
  {\centering}}
\makeatother

\makeatletter
\def\myfnt{\ifx\protect\@typeset@protect\expandafter\footnote\else\expandafter\@gobble\fi}
\makeatother

\theoremstyle{definition}

\newtheorem{theorem}{Theorem}[section]
\newtheorem{definition}[theorem]{Definition}
\newtheorem{lemma}[theorem]{Lemma}
\newtheorem{proposition}[theorem]{Proposition}
\newtheorem{example}[theorem]{Example}
\newtheorem{corollary}[theorem]{Corollary}
\newtheorem{question}[theorem]{Question}

\newcounter{claimcounter}
\numberwithin{claimcounter}{theorem}
\newenvironment{claim}{\stepcounter{claimcounter}{\noindent {\bf Claim \theclaimcounter.}}}{}
\newenvironment{claimproof}[1]{\noindent{{\em Proof.}}\space#1}{\hfill $\rule{0.35em}{0.35em}$}

\newcommand{\pureindep}[1][]{%
  \mathrel{
    \mathop{
      \vcenter{
        \hbox{\oalign{\noalign{\kern-.3ex}\hfil$\vert$\hfil\cr
              \noalign{\kern-.7ex}
              $\smile$\cr\noalign{\kern-.3ex}}}
      }
    }\displaylimits_{#1}
  }
}

\newcommand{\indep}[2]{%
  \mathrel{
    \mathop{
      \vcenter{
        \hbox{%
\oalign{
\noalign{\kern-.3ex}\hfil$\vert$\hfil\cr
              \noalign{\kern-.7ex}
              $\smile$\cr\noalign{\kern-.3ex}
}
}
      }
}^{\!\!\!\!\!#2}_{\!\!\hspace{-0.1em}#1}
  }
}

\newcommand{\displayindep}[2]{%
  \mathrel{
    \mathop{
      \vcenter{
        \hbox{%
\oalign{
\noalign{\kern-.3ex}\hfil$\vert$\hfil\cr
              \noalign{\kern-.7ex}
              $\smile$\cr\noalign{\kern-.3ex}
}
}
      }
}^{\!\!\hspace{-0.1em}#2}_{\!\!\hspace{-0.1em}#1}
  }
}

\newcommand{\displayfindep}[2]{%
  \mathrel{
    \mathop{
      \vcenter{
        \hbox{%
\oalign{
\noalign{\kern-.3ex}\hfil$\vert$\hfil\cr
              \noalign{\kern-.7ex}
              $\smile$\cr\noalign{\kern-.3ex} 
}
}
      }
}^{\!\hspace{-0.14em}#2}_{\!\!\hspace{-0.05em}#1}
  }
}

\newcommand{\spanindep}[1][]{\indep{#1}{\langle \rangle}}

\newcommand{\aclindep}[1][]{\indep{#1}{\mathrm{acl}}}
\newcommand{\clindep}[1][]{\indep{#1}{\mathrm{cl}}}
\newcommand{\forkindep}[1][]{\indep{#1}{\mathrm{f}}}

\newcommand{\ortindep}[1][]{\indep{#1}{\mathrm{ort}}}

\newcommand{\displayforkindep}[1][]{\displayfindep{#1}{\mathrm{f}}\!\hspace{-0.15em}}

\newcommand{\botc}[1]{~\bot_{#1}~}
\def\boto{\ \bot\ }

\def\presuper#1#2%
  {\mathop{}%
   \mathopen{\vphantom{#2}}^{#1}%
   \kern-\scriptspace%
   #2}

\begin{document}

\title{Independence Logic and Abstract Independence Relations}
\thanks{The research of the author was supported by the Finnish Academy of Science and Letters (Vilho, Yrj\"o and Kalle V\"ais\"al\"a foundation) and grant TM-13-8847 of CIMO. The author would like to thank above all Tapani Hyttinen, but also John Baldwin, \AA sa Hirvonen, and Jouko V\"a\"an\"anen for useful conversations related to this paper. The author would also like to thank the referee for his careful reading of the paper, his corrections and his suggestions.}

\author{Gianluca Paolini}
\address{Department of Mathematics and Statistics,  University of Helsinki, Finland}

\date{}
\maketitle

\begin{abstract} 
	
	We continue the work on the relations between independence logic and the model-theoretic analysis of independence, generalizing the results of \cite{PaoliniVaananen1} to the framework of abstract independence relations for an arbitrary $\mathrm{AEC}$. We give a model-theoretic interpretation of the independence atom and characterize under which conditions we can prove a completeness result with respect to the deductive system that axiomatizes independence in team semantics and statistics.
	
\end{abstract}


\section{Introduction}

	In mathematics and model theory the concepts of dependence and independence are of crucial importance, it is in fact always in function of an independence calculus that a classification theory for a class of classes of structures is developed. For this reason, the notions of dependence and independence are objects of intense study in the model-theoretic community. Three main frameworks in which (in)dependence has been studied are: pregeometries, first-order theories and abstract elementary classes ($\mathrm{AECs}$). Table~\ref{Independence in model theory} lists the most important cases of (in)dependence studied in these contexts. 
	
	Recently, V\"a\"an\"anen \cite{vaananen} developed a logical approach to the notions of dependence and independence, establishing a general theory of (in)dependence that goes under the name of {\em dependence logic}. Dependence logic provides an abstract characterization of (in)dependence, which accounts for the way dependence and independence behave in several disciplinary fields, e.g. database theory and statistics. 

		\begin{table}[h]
$$\begin{array}{c|c|c}
	\text{Pregeometries} & \text{Forking Indep.} & \text{Indep. in $\mathrm{AECs}$} \\
	\hline
	\hline
	\text{Vector} & \text{$\omega$-stable} & \text{$\aleph_0$-stable} \\
	\text{spaces} & \text{theories} & \text{homogeneous}  \\
		\hline
	\text{Alg. closed} & \text{Stable} & \text{Excellent} \\
	\text{fields} & \text{theories} & \text{classes}  \\
		\hline
	\text{Graphs} & \text{Simple} & \text{Finitary} \\
	\text{} & \text{theories} & \text{$\mathrm{AECs}$}  \\
	\hline
	\hline
	\end{array}
$$\caption{Independence in model theory}\label{Independence in model theory}\end{table}
	In \cite{PaoliniVaananen1} the cases of (in)dependence occurring in pregeometries and $\omega$-stable theories were also shown to be instances of this theory. We now generalize these results to the other cases of independence listed in Table~\ref{Independence in model theory}. We work in the framework of abstract independence relations for an arbitrary abstract elementary class, which subsumes most of the cases of independence of interest in model theory.
	
	The key feature of the family of logics studied in dependence logic is the presence of logical atoms different from the equational one. Each kind of atom corresponds to a different notion of (in)dependence, and each logic in the family is characterized by the logical atoms present in the syntax. This makes the study of the atomic level of the (in)dependence logics of great relevance, as indeed this is the added layer of expressivity that these systems have at disposal. This study often results in the analysis of the implication problem for a set of atoms of a particular form. That is, the search for a complete deductive system for these atoms. Emblematic examples are the axiomatizations of functional dependence and stochastic independence due to Armstrong \cite{DBLP:conf/ifip/Armstrong74} and Geiger, Paz and Pearl \cite{geiger1991axioms}, respectively.
	
	Our specific aim in this paper is the solution of the implication problem for the independence atom under a model-theoretic interpretation. This analysis was initiated in \cite{PaoliniVaananen1}, where several (in)dependence atoms were shown to have natural model-theoretic counterparts. In the present study we deal exclusively with the independence atom $\vec{x} \boto \vec{y}$ and, only marginally, with its conditional version $\vec{x} \botc{\vec{z}} \vec{y}$. 
	
	In Section~\ref{abs_indep_rel} we set the stage, defining what is an abstract elementary class and what is the axiomatization of independence to which we refer. We also give the principal examples of independence, among which {\em forking independence} in a simple theory, and {\em pregeometric independence} in an $\mathrm{AEC}$ with a uniform pregeometric operator. 
	In Section~\ref{indep_seq} we introduce a particular class of independence relations, which we call {\em federated}. We show that these are a generalization of the way independence behave in vector spaces, algebraically closed fields, and abelian groups. We then focus on its pregeometric version, and show that any $\omega$-homogenous non-trivial pregeometry is federated (modulo a finite localization). Thus, we use this result to deduce that in any first-order stable theory that admits non-trivial regular types forking independence is federated (over some set of parameters).
	In Section~\ref{indep_logic} we use the theory developed in Section~\ref{indep_seq} to characterize under which conditions we can prove a completeness result with respect to the deductive system that axiomatizes independence in team semantics and statistics, giving a complete answer to the motivating question of the paper. 
		
\section{Abstract Independence Relations}\label{abs_indep_rel}

	To make clear the levels of generalization at which we work, we first define what is an abstract independence relation in the context of first-order theories, and then generalize this definition to the context of abstract elementary classes.

\subsection{Abstract Independence Relations in First-Order Theories}

	We refer to the framework of \cite{baldwin_fun_sta_th} and \cite{adler}. We fix some notation. $AB$ is shorthand for $A \cup B$. For a complete first-order theory $T$, we denote by $\mathfrak{M}$ its monster model. 
	
	\begin{definition} Let $T$ be a complete theory and $\pureindep$ a ternary relation between (bounded) subsets of the monster model $\mathfrak{M}$. We say that $\pureindep$ is a {\em pre-independence relation} if it satisfies the following axioms.
	\begin{enumerate}[$(a.)$]
		\item (Invariance) If $A \pureindep[C] B$ and $f \in \mathrm{Aut(\mathfrak{M})}$, then $f(A) \pureindep[f(C)] f(B)$.
		\item (Existence) $A \pureindep[A] B$, for any $A, B \subseteq \mathfrak{M}$.
		\item (Monotonicity) If $A \pureindep[C] B$ and $D \subseteq A$, then $D \pureindep[C] B$.
		\item (Base Monotonicity) Let $D \subseteq C \subseteq B$. If $A \pureindep[D] B$, then $A \pureindep[C] B$.
		\item (Symmetry) If $A \pureindep[C] B$, then $B \pureindep[C] A$.
		\item (Transitivity) Let $D \subseteq C \subseteq B$. If $B \pureindep[C] A$ and $C \pureindep[D] A$, then $B \pureindep[D] A$.
		\item (Normality) If $A \pureindep[C] B$, then $AC \pureindep[C] B$.
		\item (Finite Character) If $A_0 \pureindep[C] B$ for all finite $A_0 \subseteq A$, then $A \pureindep[C] B$.
		\item (Anti-Reflexivity) If $A \pureindep[B] A$, then $A \pureindep[B] C$ for any $C \subseteq \mathfrak{M}$.

\end{enumerate}
If in addition $\pureindep$ satisfies the following two axioms, then we say that $\pureindep$ is an {\em independence relation}.
\begin{enumerate}[$(a.)$]\setcounter{enumi}{9}	
		\item (Extension) If $A \pureindep[C] B$ and $B \subseteq D$, then there is $f \in \mathrm{Aut}(\mathfrak{M})$ fixing $BC$ pointwise such that $f(A) \pureindep[C] D$.
		\item (Local Character) There is a cardinal $\kappa(T)$ such that for every $\vec{a} \in \mathfrak{M}^{< \omega}$ and $B \subseteq \mathfrak{M}$ there is $C \subseteq B$ with $|C| < \kappa(T)$ and $\vec{a} \pureindep[C] B$.

\end{enumerate}

\end{definition}

	In this context we do not distinguish between finite sets and finite sequences. Thus, if $A = \left\{ a_0, ..., a_{n-1} \right\}$, $B = \left\{ b_0, ..., b_{m-1} \right\}$ and $C = \left\{ c_0, ..., c_{k-1} \right\}$, we may write $a_0 \cdots a_{n-1} \pureindep[c_0 \cdots c_{k-1}] b_0 \cdots b_{m-1}$ instead of $A \pureindep[C] B$. By Transitivity we will refer to the following (a-priori) stronger property.

	\begin{proposition}[Transitivity]\label{trans_fo} $A \pureindep[C] B$ and $A \pureindep[CB] D$ if and only if $A \pureindep[C] BD$.
		
\end{proposition}

	\begin{proof} For the direction ($\Rightarrow$), suppose that $A \pureindep[C] BD$. We have that $A \pureindep[C] B$ by Monotonicity. Furthermore, by Symmetry and Normality we have that $A \pureindep[C] BCD$ and so, by Base Monotonicity, $A \pureindep[CB] BCD$. Thus, by Monotonicity, $A \pureindep[CB] D$. For the direction ($\Leftarrow$), suppose that $A \pureindep[C] B$ and $A \pureindep[CB] D$. By Symmetry, $B \pureindep[C] A$ and $D \pureindep[CB] A$, so, by Normality, $CB \pureindep[C] A$ and $DCB \pureindep[CB] A$. Thus, by Transitivity (the axiom), $DCB \pureindep[C] A$ and so, by Monotonicity and Symmetry, $A \pureindep[C] BD$.
		
\end{proof}

	The following principle will be of crucial importance in Section~\ref{indep_logic}.

\begin{corollary}[Exchange]\label{exchange_fo} If $A \pureindep[D] B$ and $AB \pureindep[D] C$, then $A \pureindep[D] BC$.

\end{corollary} 

\begin{proof} \[ \begin{array}{rcl}
 &\;\;\, A \pureindep[D] B  \text{ and } AB \pureindep[D] C          &  \\
        											  &  \Downarrow & \\
 &\;\;\, A \pureindep[D] B  \text{ and } C \pureindep[D] AB          &  \\
  													  &  \Downarrow & \\
 &\;\;\, A \pureindep[D] B  \text{ and } C \pureindep[DB] A        &  \;\;\;[\text{by Transitivity}] \\
													  &  \Downarrow &  \\
 &\;\;\, A \pureindep[D] B  \text{ and } A \pureindep[DB] C          & \;\;\;[\text{by Symmetry}] \\
  													  &  \Downarrow & \\
  													  &  A \pureindep[D] BC & \;\;\;[\text{by Transitivity}].
\end{array}
 \]

\end{proof}

	In the following two orthogonal examples (not generalizing each others), that cover a broad class of first-order theories. The second one is by far the most important example of independence that has ever been formulated, the original definition is due to Shelah \cite{shelah}.
	
	\begin{example}[Independence in $o$-minimal theories \cite{pillay_o-minimal}] Let $T$ be an $o$-minimal theory. For $A, B, C \subseteq \mathfrak{M}$, define $A \aclindep[B] C$ if for every $\vec{a} \in A$ we have that $\mathrm{dim}_{\mathrm{acl}}(\vec{a}/B \cup C) = \mathrm{dim}_{\mathrm{acl}}(\vec{a}/B)$. Then $\aclindep$ is a pre-independence relation.
		
\end{example}

		\begin{example}[Forking in simple theories \cite{kim}] Let $T$ be a simple theory. For $A, B, C \subseteq \mathfrak{M}$, define $A \forkindep[B] C$ if for every $\vec{a} \in A$ we have that $\mathrm{tp}(\vec{a}/B \cup C)$ is a non-forking extension of $\mathrm{tp}(\vec{a}/B)$. Then $\forkindep$ is an independence relation.
		
\end{example}

\subsection{Abstract Independence Relations in Abstract Elementary Classes}

	The axiomatization of independence that we gave in the previous section does not refer to any intrinsically first-order property, it thus makes sense to generalize it to the context of abstract elementary classes \cite{shelah_abstr_ele_cla}. First of all we define what an abstract elementary class is and what are the analog of the first-order notions of amalgamation and joint embedding.
	
	\begin{definition}[Abstract Elementary Class \cite{shelah_abstr_ele_cla}]\label{def_indep_first_order}  Let $\mathbf{K}$ be a class of structures in the vocabulary $L$ and $\preccurlyeq$ a binary relation on $\mathbf{K}$. We say that $(\mathbf{K}, \preccurlyeq)$ is an {\em abstract elementary class} ($\mathrm{AEC}$) if the following conditions are satisfied.
		\begin{enumerate}[(1)]
			\item $\mathbf{K}$ and $\preccurlyeq$ are closed under isomorphisms.
			\item If $\mathcal{A} \preccurlyeq \mathcal{B}$, then $\mathcal{A}$ is an $L$-submodel of $\mathcal{B}$.
			\item The relation $\preccurlyeq$ is a partial order on $\mathbf{K}$.
			\item If $(\mathcal{A}_i)_{i < \delta}$ is an increasing continuous $\preccurlyeq$-chain, then:
			\begin{enumerate}[({4.}1)]
				\item $\bigcup_{i < \delta} \mathcal{A}_i \in \mathbf{K}$;
				\item for each $j < \delta$, $\mathcal{A}_j \preccurlyeq \bigcup_{i < \delta} \mathcal{A}_i$;
				\item if each $\mathcal{A}_j \preccurlyeq \mathcal{B}$, then $\bigcup_{i < \delta} \mathcal{A}_i \preccurlyeq \mathcal{B}$ \; (Smoothness Axiom).
	\end{enumerate}
			\item If $\mathcal{A}, \mathcal{B}, \mathcal{C} \in \mathbf{K}$, $\mathcal{A} \preccurlyeq \mathcal{C}$, $\mathcal{B} \preccurlyeq \mathcal{C}$ and $\mathcal{A} \leq \mathcal{B}$, then $\mathcal{A} \preccurlyeq \mathcal{B}$ \; (Coherence Axiom).
			\item There is a L\"owenheim-Skolem number $\mathrm{LS}(\mathbf{K}, \preccurlyeq)$ such that if $\mathcal{A} \in \mathbf{K}$ and $B \subseteq A$, then there is $\mathcal{C} \in \mathbf{K}$ such that $B \subseteq C$, $\mathcal{C} \preccurlyeq \mathcal{A}$ and $|C| \leq |B| + |L| + \mathrm{LS}(\mathbf{K}, \preccurlyeq)$ \; (Existence of LS-number).
	
\end{enumerate}
		
\end{definition}

	\begin{definition} If $\mathcal{A}, \mathcal{B} \in \mathbf{K}$ and $f: \mathcal{A} \rightarrow \mathcal{B}$ is an embedding such that $f(\mathcal{A}) \preccurlyeq \mathcal{B}$, then we say that $f$ is a $\preccurlyeq$-embedding.
	
\end{definition}
	
	Let $\lambda$ be a cardinal. We let $\mathbf{K}_{\lambda} = \left\{ \mathcal{A} \in \mathbf{K} \; | \; |A| = \lambda \right\}$.

	\begin{definition}\label{def_AP} Let $(\mathbf{K}, \preccurlyeq)$ be an $\mathrm{AEC}$. 
		\begin{enumerate}[(i)]
			\item We say that $(\mathbf{K}, \preccurlyeq)$ has the {\em amalgamation property} $(\mathrm{AP})$ if for any $\mathcal{A}, \mathcal{B}_0, \mathcal{B}_1 \in \mathbf{K}$ with $\mathcal{A} \preccurlyeq \mathcal{B}_i$ for $i < 2$, there are $\mathcal{C} \in \mathbf{K}$ and $\preccurlyeq$-embeddings $f_i: \mathcal{B}_i \rightarrow \mathcal{C}$ for $i < 2$, such that $f_0 \restriction A = f_1 \restriction A$.
			\item We say that $(\mathbf{K}, \preccurlyeq)$ has the {\em joint embedding property} $(\mathrm{JEP})$ if for any $\mathcal{B}_0, \mathcal{B}_1 \in \mathbf{K}$ there are $\mathcal{C} \in \mathbf{K}$ and $\preccurlyeq$-embeddings $f_i: \mathcal{B}_i \rightarrow \mathcal{C}$ for $i < 2$.
			\item We say that $(\mathbf{K}, \preccurlyeq)$ has arbitrarily large models $(\mathrm{ALM})$ if for every $\lambda \geq \mathrm{LS}(\mathbf{K})$, $\mathbf{K}_{\lambda} \neq \emptyset$.
			
	\end{enumerate}
		
\end{definition}

	If $(\mathbf{K}, \preccurlyeq)$ has $\mathrm{AP}$, $\mathrm{JEP}$ and $\mathrm{ALM}$, then, using the same technique as in the elementary case, we can build a monster model for $(\mathbf{K}, \preccurlyeq)$. Consistent with the notation used for the elementary case, we denote this model by $\mathfrak{M}$. We are now in the position to generalize Definition~\ref{def_indep_first_order} to the context of $\mathrm{AECs}$. Also in this case we distinguish between pre-independence and independence relations. In our study we will work only at the level of pre-independence, but we consider worth mentioning what are (some of) the further axioms that are required in order to develop a classification theory for the $\mathrm{AEC}$ under examination. 
	
	\begin{definition} Let $(\mathbf{K}, \preccurlyeq)$ be an $\mathrm{AEC}$ with $\mathrm{AP}$, $\mathrm{JEP}$ and $\mathrm{ALM}$, and $\pureindep$ a ternary relation between (bounded) subsets of the monster model $\mathfrak{M}$. We say that $\pureindep$ is a {\em pre-independence relation} if it satisfies the following axioms.
		\begin{enumerate}[$(a.)$]
			\item (Invariance) If $A \pureindep[C] B$ and $f \in \mathrm{Aut(\mathfrak{M})}$, then $f(A) \pureindep[f(C)] f(B)$.
			\item (Existence) $A \pureindep[A] B$, for any $A, B \subseteq \mathfrak{M}$.
			\item (Monotonicity) If $A \pureindep[C] B$ and $D \subseteq A$, then $D \pureindep[C] B$.
			\item (Base Monotonicity) Let $D \subseteq C \subseteq B$. If $A \pureindep[D] B$, then $A \pureindep[C] B$.
			\item (Symmetry) If $A \pureindep[C] B$, then $B \pureindep[C] A$.
			\item (Transitivity) Let $D \subseteq C \subseteq B$. If $B \pureindep[C] A$ and $C \pureindep[D] A$, then $B \pureindep[D] A$.
			\item (Normality) If $A \pureindep[C] B$, then $AC \pureindep[C] B$.
			\item (Finite Character) If $A_0 \pureindep[C] B$ for all finite $A_0 \subseteq A$, then $A \pureindep[C] B$.
			\item (Anti-Reflexivity) If $A \pureindep[B] A$, then $A \pureindep[B] C$ for any $C \subseteq \mathfrak{M}$.
			
\end{enumerate}
	If in addition $\pureindep$ satisfies the following two axioms, then we say that $\pureindep$ is an {\em independence relation}.
	\begin{enumerate}[$(a.)$]\setcounter{enumi}{9}	
					\item (Extension) If $A \pureindep[C] B$ and $B \subseteq D$, then there is $f \in \mathrm{Aut}(\mathfrak{M})$ fixing $BC$ pointwise such that $f(A) \pureindep[C] D$.
					\item (Local Character) There is a cardinal $\kappa(\mathbf{K})$ such that for every $\vec{a} \in \mathfrak{M}^{< \omega}$ and $B \subseteq \mathfrak{M}$ there is $C \subseteq B$ with $|C| < \kappa(\mathbf{K})$ and $\vec{a} \pureindep[C] B$.

		\end{enumerate}
		
\end{definition}

	As in the previous section, by Transitivity we will refer to the following (a-priori) stronger property.
	
	\begin{proposition}[Transitivity] $A \pureindep[C] B$ and $A \pureindep[CB] D$ if and only if $A \pureindep[C] BD$.

\end{proposition} 

	\begin{proof} As in Proposition~\ref{trans_fo}.
		
\end{proof}

	\begin{corollary}[Exchange] If $A \pureindep[D] B$ and $AB \pureindep[D] C$, then $A \pureindep[D] BC$.

\end{corollary} 

	\begin{proof} As in Corollary~\ref{exchange_fo}.
		
\end{proof}
	
	If $T$ is a complete first-order theory and we denote by $\preccurlyeq$ the relation of elementary substructure, then the pair $(\mathbf{Mod}(T), \preccurlyeq)$ is an $\mathrm{AEC}$ with $\mathrm{AP}$, $\mathrm{JEP}$ and $\mathrm{ALM}$. Thus all the cases of independence examined in the previous section are instances of this more general definition. Furthermore, the generality at which we work allow us to subsume also the non-elementary cases of independence.
	
	\begin{example}[Independence in Pregeometries \cite{grossberg_lessmann}] Let $(\mathbf{K}, \preccurlyeq)$ be an $\mathrm{AEC}$ with $\mathrm{AP}$, $\mathrm{JEP}$ and $\mathrm{ALM}$, and $\mathrm{cl}: \mathfrak{M} \rightarrow \mathfrak{M}$ a pregeometric operator. For $A, B, C \subseteq \mathfrak{M}$, define $A \clindep[B] C$ if for every $\vec{a} \in A$ we have $\mathrm{dim}_{\mathrm{cl}}(\vec{a}/B \cup C) = \mathrm{dim}_{\mathrm{cl}}(\vec{a}/B)$. Then $\clindep$ is a pre-independence relation.
		
\end{example}

	\begin{example}[Hilbert Spaces]\label{hilbert_spaces} Let $\mathbf{K}$ be the class of Hilbert Spaces over $\mathbb{R}$ (resp. $\mathbb{C}$) and $\preccurlyeq$ the closed linear subspace relation, then $(\mathbf{K}, \preccurlyeq)$ is an $\mathrm{AEC}$ with $\mathrm{AP}$, $\mathrm{JEP}$ and $\mathrm{ALM}$. Given a closed linear subspace $C \subseteq \mathfrak{M}$ and $a \in \mathfrak{M}$, we denote by $\mathrm{P}_C (a)$ the orthogonal projection of $a$ onto $C$. For $D, B \subseteq \mathfrak{M}$, we then say that $D \ortindep[A] B$ if for every $a \in D$ and $b \in B$ we have $\mathrm{P}_{A^{\bot}} (a) \boto \mathrm{P}_{A^{\bot}} (b)$. Then $\ortindep$ is orthogonal over $A$ for any $A \subseteq \mathfrak{M}$. 

\end{example}
	
	\begin{example}[Independence in Finitary $\mathrm{AECs}$] See \cite{hyttinen_and_kesala}.
		
\end{example}

\section{Federation}\label{indep_seq}

	We introduce two fundamental notions: {\em independent sequences} and {\em algebraic tuples}. Independent sequences play a fundamental role in classification theory, where they often occur in the form of sequences of indiscernibles.

	\begin{definition}[Independent Sequence] Let $\pureindep$ be a pre-independence relation and $(I, <)$ a linear order. Let $A \subseteq \mathfrak{M}$ and $(a_i \; | \; i \in I) \in \mathfrak{M}^{I}$ injective. We say that $(a_i \; | \; i \in I)$ is an $\pureindep$-{\em independent sequence} over $A$ if for all $j \in I$, we have $( a_i \; | \; i < j ) \pureindep[A] a_j$. We say that $(a_i \; | \; i \in I)$ is an $\pureindep$-independent sequence if it is an $\pureindep$-independent sequence over $\emptyset$.
	
\end{definition}

	\begin{definition}[Algebraic Tuple]	Let $\pureindep$ a pre-independence relation. We say that $\vec{e} \in \mathfrak{M}^{< \omega}$ is $\pureindep$-{\em algebraic} over $A$ if $\vec{e} \pureindep[A] \vec{e}$. We say that $\vec{e} \in \mathfrak{M}^{< \omega}$ is $\pureindep$-algebraic if it is $\pureindep$-algebraic over $\emptyset$.
		
\end{definition}

	When it is clear to which pre-independence relation $\pureindep$ we refer, we just talk of independent sequences and algebraic tuples. 

	\begin{lemma}\label{lemma_indep_seq} Let $\pureindep$ be a pre-independence relation and $(a_i \; | \; i \in I) \in \mathfrak{M}^{I}$ an independent sequence over $A$. Then for all $\vec{a}, \vec{b} \subseteq (a_i \; | \; i \in I) \in \mathfrak{M}^{I}$ with $\vec{a} \cap \vec{b} = \emptyset$ we have $\vec{a} \pureindep[A] \vec{b}$.
		
\end{lemma}

	\begin{proof} It suffices to show that for $\vec{a} = (a_{k_0}, ..., a_{k_{n-1}})$ and $\vec{b} = (a_{j_0}, ..., a_{j_{m-1}})$ with $k_0 < \cdots < k_{n-1}$, $j_0 < \cdots < j_{m-1}$ and $\vec{a} \cap \vec{b} = \emptyset$, we have that $\vec{a} \pureindep[A] \vec{b}$. We prove this by induction on $\mathrm{max}(k_{n-1}, j_{m-1}) = t$. 
\newline $t = 0$. If this is the case, then either $\overline{a} = \emptyset$ or $\overline{b} = \emptyset$ because $\vec{a} \cap \vec{b} = \emptyset$. Suppose the first, the other case is symmetrical. By Existence $A \pureindep[A] \overline{b}$, and so, by Monotonicity, $\emptyset \pureindep[A] \overline{b}$.
\newline $t > 0$. Suppose that $t = j_{m-1}$, the other case is symmetrical. By the independence of the sequence and Monotonicity, it follows that 
	\[ a_{k_0} \cdots a_{k_{n-1}} b_{j_0} \cdots b_{j_{m-2}} \pureindep[A] b_{j_{m-1}}.\]
	Notice now that $\mathrm{max}(k_{n-1}, j_{m-2}) < t$ because $\vec{a} \cap \vec{b} = \emptyset$, thus by induction hypothesis we have that
	\[ a_{k_0} \cdots a_{k_{n-1}} \pureindep[A] b_{j_0} \cdots b_{j_{m-2}}. \]
	Hence by Exchange we can conclude that 
	\[ a_{k_0} \cdots a_{k_{n-1}} \pureindep[A] b_{j_0} \cdots b_{j_{m-1}}. \]
		
\end{proof}

	\begin{corollary} Let $\pureindep$ be a pre-independence relation and $(a_i \; | \; i \in I) \in \mathfrak{M}^{I}$ be an independent sequence over $A$, then for every $\vec{a} \subseteq (a_i \; | \; i \in I)$, we have $\vec{a} \pureindep[A] ( a_i \; | \; i \in I ) - \vec{a}$.
		
\end{corollary}

	\begin{proof} Follows from Lemma~\ref{lemma_indep_seq} by Finite Character.
		
\end{proof}

	We define the notion of {\em federation}. This notion is a generalization of the notion of federated pregeometry introduced in \cite{baldwin_abs_dep_rel}. For an independent sequence to be federated we ask the existence of a point which is {\em dependent} from all the members of the sequence, but independent from all but one. It can be thought as a strong form of independence. We denote by $\omega^*$ the set $\omega - \left\{ 0 \right\}$.

	\begin{definition}[Federation] Let $\pureindep$ be a pre-independence relation, $n \in \omega^*$ and $(a_i \, | \, i < n) \in \mathfrak{M}^n$ an independent sequence over $A$. We say that $(a_i \, | \, i < n)$ is federated over $A$ if there exists $d \in \mathfrak{M}$ such that
		\[ d \not\!\pureindep[A] a_0 \cdots a_{n-1} \;\; \text{ and } \;\; d \, \pureindep[A] (a_i \, | \, i < n ) - a_j \,  \text{ for every $j < n$}. \] 
		We say that $(a_i \, | \, i < n)$ is federated if it is federated over $\emptyset$. For $(a_i \, | \, i < \omega) \in \mathfrak{M}^\omega$ independent (over $A$), we say that $(a_i \, | \, i < \omega)$ is federated (over $A$) if $(a_i \, | \, i < n)$ is federated (over $A$) for every $n \in \omega^*$.
		
\end{definition}
	
\begin{definition} Let $\pureindep$ be a pre-independence relation and $A \subseteq \mathfrak{M}$. We define the index of federation of $\pureindep$ over $A$, in symbols $\mathrm{IF}(\pureindep; A)$, as 
	\[\mathrm{sup}\left\{ n \in \omega^* \, | \, \text{ there is } (a_i \, | \, i < n) \in \mathfrak{M}^n \text{ federated over } A \right\}. \]
	We say that $\pureindep$ is {\em federated} over $A$ if $\mathrm{IF}(\pureindep; A) = \omega$. We say that $\pureindep$ is federated if it is federated over $\emptyset$.
\end{definition}

	Clearly, the easiest way to show that a particular pre-independence relation is federated is to find a federated sequence of length $\omega$ in the monster model. This will be our way to establish the federation of a pre-independence relation. 
	
\subsection{Federated Pregeometries}\label{alg_pregeo}

	In the following three important examples of federated independent relations.

	\begin{example}[Vector spaces \cite{PaoliniVaananen1}] Let $\mathrm{VS}^{\mathrm{inf}}_{\mathbb{K}}$ denote the theory of infinite vector spaces over a fixed field $\mathbb{K}$. Let $\langle \rangle: \mathfrak{M} \rightarrow \mathfrak{M}$ be such that $A \mapsto \langle A \rangle$, i.e. the linear span of $A$, then $\langle \rangle$ is a pregeometric operator. Notice that the theory $\mathrm{VS}^{\mathrm{inf}}_{\mathbb{K}}$ is superstable (if $\mathbb{K}$ is countable it actually is $\omega$-stable) and strongly minimal. Furthermore, the span operator coincides with the algebraic closure operator. Thus in this case we have that $\spanindep \, = \, \aclindep \, = \, \forkindep$. Let $A \subseteq \mathfrak{M}$ be such that $\mathrm{dim}(A) = \aleph_0$ and let $(a_i \; | \; i \in \omega)$ be an injective enumeration of a basis $B$ for $A$ in $\mathfrak{M}$, then $(a_i \; | \; i \in \omega)$ is a federated sequence. Notice that $0 \in \mathfrak{M}$ is an algebraic point.
	
\end{example}

	\begin{example}[Algebraically closed fields \cite{PaoliniVaananen1}] Let $\mathrm{ACF}_p$ denote the theory of algebraically closed fields of characteristic $p$, where $p$ is either $0$ or a prime. Let $\mathrm{acl}: \mathfrak{M} \rightarrow \mathfrak{M}$ be such that $A \mapsto \mathrm{acl}(A)$, i.e. the algebraic closure of $A$, then $\mathrm{acl}$ is a pregeometric operator. Notice that the theory $\mathrm{ACF}_p$ is $\omega$-stable and, furthermore, it is strongly minimal, thus in this case we have that $\aclindep \, = \, \forkindep$. Let $A \subseteq \mathfrak{M}$ be such that $\mathrm{dim}(A) = \aleph_0$ and let $(a_i \; | \; i \in \omega)$ be an injective enumeration of a basis $B$ for $A$ in $\mathcal{K}$, then $(a_i \; | \; i \in \omega)$ is a federated sequence. Notice that any member of the prime field of $\mathfrak{M}$ is an algebraic point.
	
\end{example}

 	\begin{example}[Abelian groups]\label{pure_subgroup_pregeometry} Let $\mathbf{K}$ be the class of abelian groups. Given $\mathcal{G}, \mathcal{H} \in \mathbf{K}$ we say that $\mathcal{G}$ is a {\em pure subgroup} of $\mathcal{H}$ if $\mathcal{G}$ is a subgroup of $\mathcal{H}$ and for every $g \in G$ and $n < \omega$, the equation $nx = g$ is solvable in $\mathcal{G}$, whenever it is solvable in $\mathcal{H}$.  Let $\preccurlyeq_{\mathrm{pure}}$ be the pure subgroup relation, then the class $(\mathcal{K}, \preccurlyeq_{\mathrm{pure}})$ is an $\mathrm{AEC}$ with $\mathrm{AP}$, $\mathrm{JEP}$ and $\mathrm{ALM}$. Let $\langle \rangle_P: \mathfrak{M} \rightarrow \mathfrak{M}$ be such that $A \mapsto \langle A \rangle_P = \left\{ b \in \mathfrak{M} \; | \; \exists n \in \omega^* \text{ with } \; nb \in \langle A \rangle \right\}$, i.e. the pure subgroup generated by $A$, then $\langle \rangle_P$ is a pregeometric operator. Let $A \subseteq \mathfrak{M}$ be such that $\mathrm{dim}(A) = \aleph_0$ and let $(a_i \; | \; i \in \omega)$ be an injective enumeration of a basis $B$ for $A$ in $\mathfrak{M}$, then $(a_i \; | \; i \in \omega)$ is a federated sequence. Notice that $0 \in \mathfrak{M}$ is an algebraic point.
		
\end{example}

	The three examples above are instances of a general pregeometric phenomenon, namely {\em federation}. This is the notion considered in \cite{baldwin_abs_dep_rel}, which we generalized to an arbitrary pre-independence relation.

	\begin{definition}[Federated Pregeometry] Let $(X, \mathrm{cl})$ be a pregeometry. We say that the pregeometry is {\em federated} if for every independent $D_0 \subseteq_{\omega} X$, $\mathrm{cl}(D_0) \neq \bigcup_{D \subsetneq D_0} \mathrm{cl}(D)$.

\end{definition}

	In infinite dimensional federated pregeometries we can always find federated sequences of length $\omega$.

	\begin{example} Let $(\mathbf{K}, \preccurlyeq)$ be an $\mathrm{AEC}$ with $\mathrm{AP}$, $\mathrm{JEP}$ and $\mathrm{ALM}$, and $\mathrm{cl}: \mathfrak{M} \rightarrow \mathfrak{M}$ a pregeometric operator such that it determines a federated pregeometry. Let $A \subseteq \mathfrak{M}$ be such that $\mathrm{dim}(A) = \aleph_0$, and $(a_i \; | \; i \in \omega)$ an injective enumeration of a basis $B$ for $A$ in $\mathfrak{M}$, then $(a_i \; | \; i \in \omega)$ is a federated sequence. Notice that if there exists $e \in \mathrm{cl}(\emptyset)$, then $e$ is an algebraic point.
		
\end{example}

	We conclude this section with an important characterization of federated pregeometries.
	
	\begin{definition} Let $(X, \mathrm{cl})$ be a pregeometry.
		
		\begin{enumerate}[i)]
			\item  We say that $(X, \mathrm{cl})$ is trivial if $\mathrm{cl}(A) = \bigcup_{a \in A} \mathrm{cl}(\left\{ a \right\})$ for every $A \subseteq X$.
			\item  We say that $(X, \mathrm{cl})$ is {\em $\omega$-homogeneous} if for every $A \subseteq_{\omega} X$ and $a, b \in X - \mathrm{cl}(A)$ there is $f \in \mathrm{Aut((X, \mathrm{cl}) / A)}$ such that $f(a) = b$.
		\end{enumerate}
		
\end{definition}

	Clearly federated pregeometries are non-trivial, more interestingly under the assumption of $\omega$-homogeneity we also have the following partial converse.
	
\makeatletter 	
\begingroup
\apptocmd{\thetheorem}{\unless\ifx\protect\@unexpandable@protect\protect\footnote{This theorem is due to Tapani Hyttinen. The proof is given with his permission.}\fi}{}{}

	\begin{theorem}[Hyttinen]\label{non-triviality_implies_algebraicity} Let $(X, \mathrm{cl})$ be an $\omega$-homogeneous pregeometry. If $(X, \mathrm{cl})$ is non-trivial, then there exists $A_0 \subseteq_{\omega} X$ such that $(X, \mathrm{cl}_{A_0})$ is federated.
	
\end{theorem} 
\endgroup
\makeatother 

	\begin{proof} Suppose that $(X, \mathrm{cl})$ is non-trivial, then there is $A \subseteq X$ such that $d \in \mathrm{cl}(A)$ but $d \notin \bigcup_{a \in A} \mathrm{cl}(\left\{ a \right\})$. By Finite Character, there is $A^{**} \subseteq_{\omega} A$, such that
		\[\tag{$\star$} d \in \mathrm{cl}(A^{**}) \text{ but } d \notin \bigcup_{a \in A^{**}} \mathrm{cl}(\left\{ a \right\}). \]
	Let $A^* = \left\{ a_0, ..., a_{n-1} \right\} \subseteq A^{**}$ be of minimal cardinality with respect to property $(\star)$, then we must have that $d \in \mathrm{cl}(\left\{ a_0, ..., a_{n-1} \right\})$, but 
	\[ d \notin \mathrm{cl}(\left\{ a_0, ..., a_{n-3} \right\} \cup \left\{ a_{n-2} \right\}) \text{ and } d \notin \mathrm{cl}(\left\{ a_0, ..., a_{n-3} \right\} \cup \left\{ a_{n-1} \right\}). \]
	Let $A_0 = \left\{ a_0, ..., a_{n-3} \right\}$, we claim that $(X, \mathrm{cl}_{A_0})$ is federated. For ease of notation, for $a, b \in X$ instead of $\mathrm{cl}(\left\{ a, b \right\})$ we just write $\mathrm{cl}(a, b)$, and analogously for singletons. Let $D_0 = \left\{ d_0, ..., d_{m-1} \right\}$ be independent in $(X, \mathrm{cl}_{A_0})$. By induction on $m$ we construct $d^{*}_0, ..., d^{*}_{m-1} \in X$ such that for $i < m -1$:
	\begin{enumerate}[i)]
		\item $d^*_{i} \notin \mathrm{cl}_{A_0}(d_{i+1})$ and $d_{i+1} \notin \mathrm{cl}_{A_0}(d^*_{i})$;
	\end{enumerate}
and, for $1 \leq i < m$:	
	\begin{enumerate}[i)]\setcounter{enumi}{1}
		\item $d^*_{i} \in \mathrm{cl}_{A_0}(d^*_{i-1}, d_i) - (\mathrm{cl}_{A_0}(d^*_{i-1}) \cup \mathrm{cl}_{A_0}(d_i))$;
		\item $d^*_{i-1} \in \mathrm{cl}_{A_0}(d^*_i, d_i) - (\mathrm{cl}_{A_0}(d^*_i) \cup \mathrm{cl}_{A_0}(d_i))$;
		\item $d_i \in \mathrm{cl}_{A_0}(d^*_{i-1}, d^*_i) - (\mathrm{cl}_{A_0}(d^*_{i-1}) \cup \mathrm{cl}_{A_0}(d_i^*))$.
\end{enumerate}
By properties ii) - iv) it will then be clear that $d^{*}_{m-1} \in  \mathrm{cl}(D_0) - \bigcup_{D \subsetneq D_0} \mathrm{cl}(D)$, as wanted. We start the construction. Let $d^{*}_0 = d_0$. Suppose then that we have defined $d^{*}_i$, we want to define $d^{*}_{i+1}$. We notice the following:
	\begin{enumerate}[1)]
		\item $a_{n-2} \notin \mathrm{cl}_{A_0}(\emptyset)$ and $d^*_{i} \notin \mathrm{cl}_{A_0}(\emptyset)$;
		\item $a_{n-1} \notin \mathrm{cl}_{A_0}(a_{n-2})$ and $d_{i+1} \notin \mathrm{cl}_{A_0}(d^*_{i})$;
		\item $d \in \mathrm{cl}_{A_0}(a_{n-1}, a_{n-2}) - (\mathrm{cl}_{A_0}(a_{n-2}) \cup \mathrm{cl}_{A_0}(a_{n-1}))$.
		
\end{enumerate}
Because of $\omega$-homogeneity and 1) we can find $f_1 \in \mathrm{Aut}((X, \mathrm{cl}) / A_0)$ such that $f_1(a_{n-2}) = d^*_{i}$. But then by 2) we have $f_1(a_{n-1}) \notin (\mathrm{cl}_{A_0}(d^*_{i}))$, and so again by $\omega$-homogeneity we can find $f_2 \in \mathrm{Aut}((X, \mathrm{cl}) / A_0 \cup \left\{ d^*_{i} \right\})$ such that
\[ a_{n-2} \xmapsto{f_1} d^*_{i} \xmapsto{f_2} d^*_{i} \]
\[ a_{n-1} \xmapsto{f_1} f_1(a_{n-1}) \xmapsto{f_2} d_{i+1}. \]
Let $d^{*}_{i+1} = f_2(f_1(d))$. We show that $d^{*}_{i+1}$ has properties i) - iv). Property ii) is clear from 3), and properties iii) and iv) follow from ii) by Exchange. Regarding property i), if $i = m-1$ there is nothing to prove. Suppose then that $i < m-1$. If $d_{i+2} \in \mathrm{cl}_{A_0}(d^*_{i+1})$, then $d_{i+2} \in \mathrm{cl}_{A_0}(\left\{ d_j \, | \, j < i+2 \right\})$, because $d^*_{i+1} \in \mathrm{cl}_{A_0}(\left\{ d_j \, | \, j < i+2 \right\})$, hence we contradict the independence of $D_0$. Finally, suppose that $d^*_{i+1} \in \mathrm{cl}_{A_0}(d_{i+2})$. By the already proved ii) for $i+1$ it follows in particular that $d^*_{i+1} \notin \mathrm{cl}_{A_0}(\emptyset)$, hence by Exchange $d_{i+2} \in \mathrm{cl}_{A_0}(d^*_{i+1})$, and so we are in the case just considered, which leads to a contradiction.
		
\end{proof}

\subsection{Federated Sequences in Stable Theories\myfnt{The author would like to thank Tapani Hyttinen for the help in the writing of this section.}}

	We are driven by the following questions.

	\begin{question}\label{question_1} Let $\mathcal{G}$ be an $\omega$-stable (resp. superstable and stable) group. Can we find $\forkindep$-federated sequences in the monster model for $\mathrm{Th}(\mathcal{G})$? Under which conditions is an $\forkindep$-independent sequence a $\forkindep$-federated sequence?
	
\end{question}

	\begin{question}\label{question_2} Are there known classes of theories in which we can always find $\forkindep$-federated sequences? Under which conditions can we find $\forkindep$-federated sequences in the stability-theoretic classes of theories, e.g. classifiable or stable?

\end{question}

		We give a complete answer to Question~\ref{question_1} and a partial answer to Question~\ref{question_2}.

	\begin{proposition} Let $\mathcal{G}$ be a stable group, then in the monster model for $\mathrm{Th}(\mathcal{G})$ we can find a $\forkindep$-federated sequence. In fact, any $\forkindep$-independent sequence of generic elements is federated.

	\end{proposition} 

	\begin{proof} Let $(a_i \, | \, i \in \omega)$ be such that $a_i$ realizes a generic type over $(a_j \, | \, j < i)$, for every $i < \omega$. Then for every $n \in \omega^*$ we have that
		\[ \sum_{i < n} a_i \not\!\displayforkindep[\emptyset] a_0 \cdots a_{n-1} \;\; \text{ and } \;\; \sum_{i < n} a_i \, \displayforkindep[\emptyset] ( a_i \, | \, i < n) - a_j  \, \text{ for every $j < n$}. \] 

\end{proof}

	\begin{theorem}\label{regular_types_and_pregeo} Let $T$ be a stable theory, $A \subseteq \mathfrak{M}$, $p \in \mathrm{S}_n(A)$ a regular type and $X \subseteq \mathfrak{M}^{n}$ the set of realizations of $p$ in $\mathfrak{M}$. Then on $X$ the forking dependence relation determines an infinite dimensional $\omega$-homogenous pregeometry $(X, \mathrm{cl}^{\mathrm{f}})$. 
		
\end{theorem} 

	\begin{proof} See for example \cite{baldwin_fun_sta_th}.
		
\end{proof}

	From Theorems~\ref{non-triviality_implies_algebraicity} and \ref{regular_types_and_pregeo} it follows directly the following corollary, which ensures that if the theory admits non-trivial regular types then we can always find federated sequences (over some set of parameters).

	\begin{corollary} Let $T$ be a stable theory, $A \subseteq \mathfrak{M}$, $p \in \mathrm{S}_1(A)$ a regular type and $X \subseteq \mathfrak{M}$ the set of realizations of $p$ in $\mathfrak{M}$. If $(X, \mathrm{cl}^{\mathrm{f}})$ is non-trivial, then we can find $A_0 \subseteq_{\omega} X$ and $(a_i \, | \, i \in \omega) \in X^{\omega}$ such that $(a_i \, | \, i \in \omega)$ is a $\forkindep$-federated sequence over $A \cup A_0$.
	
\end{corollary} 

	\begin{proof} By Theorem~\ref{non-triviality_implies_algebraicity} there is $A_0 \subseteq_{\omega} X$ such that $(X, \mathrm{cl}_{A_0})$ is federated. Notice that the pregeometry $(X, \mathrm{cl}_{A_0})$ is also infinite dimensional. Let $(a_i \, | \, i \in \omega)$ be an enumeration of the fist $\omega$ elements in a basis $B$ for $(X, \mathrm{cl}_{A_0})$. Then $(a_i \, | \, i \in \omega)$ is a $\forkindep$-federated sequence over $A \cup A_0$.
	
\end{proof}

\section{Independence Logic}\label{indep_logic}

	We now enter in the dependence logic component of the paper. In the first section we describe how the independence atom is characterized in {\em team semantics} and study its axiomatization. Team semantics is a new semantic tool introduced in \cite{hodges} and then developed in \cite{vaananen}, which is based on the idea of giving semantics to logic languages by means of sets of assignments instead of single assignments. In the second section we interpret the independence atom as $\pureindep$ and study the implication problem for the resulting system. Notice that all the logic systems described in this paper have an atomic language, with no connectives and no quantifiers. It may also be worth noticing that the atomic systems described in Section~\ref{atomic_indep_logic} are part of a wider logic language with actual logical operations $\wedge, \vee, \neg, \exists, \forall$. For details see \cite{vaananen}. In the case of the systems described in Section~\ref{abstr_indep_rel_indep_logic}, is it not yet clear what would be the right way to extend the atomic system to a logic with connectives and quantifiers. Further investigations will probably answer this question. 
	
	In the case of the systems based on team semantics we give a semantics based on first-order structures, i.e. the usual structures with respect to which first-order logic is defined. In the present treatment of the subject we do not consider (non-logical) predicates and terms, and so it may seem weird that we give the semantics with respect to structures instead of mere sets. As a matter of facts, we could have used structureless sets instead of structures. We use structures to stress that these systems are fragments of a broader language, where predicates and terms play a natural role. In the present paper the focus is on the independence phenomenon, which manifests itself at the atomic level, and it is independent from the fact that we allow terms to occur in the independence atoms. 
	
	Another point worth making explicit is about the systems described in Section~\ref{abstr_indep_rel_indep_logic}. These systems are defined with respect to a fixed $\mathrm{AEC}$ and a pre-independence relation $\pureindep$ on its monster model. The dependence on the particular pre-independence relation considered is crucial. This is made clear by the content of Theorems~\ref{completeness_indep} and \ref{characterization}. While the unconditional system is sound independently of the choice of $\pureindep$, we have a completeness result with respect to the deductive system described in Section~\ref{atomic_indep_logic} if and only if the pre-independence relation under investigation fulfills some specific requirements, namely federation and admissibility of an algebraic point.
	
	\subsection{Atomic Independence Logic}\label{atomic_indep_logic}

	 \index{Atomic!Independende Logic} {\em Atomic Independendce Logic} ($\mathrm{AIndL}$) is defined as follows.
	%
	%
		The language of this logic is made only of independence atoms. That is, let $\vec{x}$ and $\vec{y}$ be finite sequences of variables, then the formula $\vec{x} \boto \vec{y}$ is a formula of the language of $\mathrm{AIndL}$.
	%
	%
		The intuitive meaning of the atom $\vec{x} \boto \vec{y}$ in team semantics is that the values of the variables in $\vec{x}$ give no information about the values of the variables in $\vec{y}$ and vice versa. The semantics is defined as in \cite{GV12}.
		We denote by $\mathrm{Var}$ the set of first-order variables. Let $\mathcal{M}$ be a first order structure. Let $X = \left\{ s_i \right\}_{i\in I}$ with $s_i: \mathrm{dom}(X) \rightarrow M$ and $\vec{x}  \vec{y} \subseteq \mathrm{dom}(X) \subseteq \mathrm{Var}$. We say that $\mathcal{M}$ satisfies $\vec{x} \boto \vec{y}$ under $X$, in symbols $\mathcal{M} \models_X \vec{x} \boto \vec{y}$, if
		\[ \forall s, s' \in X \;  \exists s'' \in X \; (s''(\vec{x}) =s(\vec{x}) \wedge s''(\vec{y}) = s'(\vec{y})). \]
	%
	%
	Let $\Sigma$ be a set of atoms and let $X$ be such that the set of variables occurring in $\Sigma$ is included in $\mathrm{dom}(X)$. We say that $\mathcal{M}$ satisfies $\Sigma$ under $X$, in symbols $\mathcal{M} \models_X \Sigma$, if $\mathcal{M}$ satisfies every atom in $\Sigma$ under $X$. 
	%
	We say that $\vec{x} \boto \vec{y}$ is a logical consequence of $\Sigma$, in symbols $\Sigma \models \; \vec{x} \boto \vec{y}$, if for every $\mathcal{M}$ and $X$ such that the set of variables occurring in $\Sigma \cup \left\{ \vec{x} \boto \vec{y} \right\}$ is included in $\mathrm{dom}(X)$ we have that
			\[ \text{ if } \; \mathcal{M} \models_X \Sigma \; \text{ then } \; \mathcal{M} \models_X \vec{x} \boto \vec{y}. \]
	%

		The deductive system of $\mathrm{AIndL}$ consists of the following rules:
	\begin{enumerate}[($a_3.$)]
		\item $\vec{x} \boto \emptyset$;
		\item If $\vec{x} \boto \vec{y}$, then $\vec{y} \boto \vec{x}$;
		\item If $\vec{x} \boto \vec{y}\vec{z}$, then $\vec{x} \boto \vec{y}$;
		\item If $\vec{x} \boto \vec{y}$ and $\vec{x}\vec{y} \boto \vec{z}$, then $\vec{x} \boto \vec{y}\vec{z}$;
		\item If $x \boto x$, then $x \boto \vec{y}$ [for arbitrary $\vec{y}$];
		\item If $\vec{x} \boto \vec{y}$, then $\pi\vec{x} \boto \sigma\vec{y}$ [where $\pi$ and $\sigma$ are permutations of $\vec{x}$ and $\vec{y}$ respectively];
		\item If $\vec{x} y \vec{z} \boto \vec{w}$, then $\vec{x} y y \vec{z} \boto \vec{w}$.

	\end{enumerate}
	A deduction from a set of atoms $\Sigma$ is a sequence of atoms $(\phi_0 , ... , \phi_{n-1})$ such that each $\phi_i$ is either an element of $\Sigma$, an instance of axiom ($a{_3.}$), or follows from one or more formulas of $\Sigma \cup \left\{ \phi_0, ... , \phi_{i-1} \right\}$ by one of the rules presented above. We say that $\phi$ is provable from $\Sigma$, in symbols $\Sigma \vdash \phi$, if there is a deduction $(\phi_0 , ... , \phi_{n-1})$ from $\Sigma$ with $\phi = \phi_{n-1}$.

		\begin{theorem}[\cite{geiger1991axioms} and \cite{galliani_and_vaananen}] Let $\Sigma$ be a set of atoms, then 
	\[\Sigma \models \vec{x} \boto \vec{y} \text{ if and only if }  \Sigma \vdash \vec{x} \boto \vec{y}. \]

	\end{theorem}

		\index{Atomic!Conditional Independence Logic} {\em Atomic Conditional Independence Logic} ($\mathrm{ACIndL}$) is defined as follows.
	%
	%
		The language of this logic is made only of conditional independence atoms. That is, let $\vec{x}$, $\vec{y}$  and $\vec{z}$ be finite sequences of variables, then the formula $\vec{x} \botc{\vec{z}} \vec{y}$ is a formula of the language of $\mathrm{ACIndL}$.
	%
	%
	The semantics is defined as in \cite{GV12}. 
		Let $\mathcal{M}$ be a first order structure. Let $X = \left\{ s_i \right\}_{i\in I}$ with $s_i: \mathrm{dom}(X) \rightarrow M$ and $\vec{x}  \vec{y}  \vec{z} \subseteq \mathrm{dom}(X) \subseteq \mathrm{Var}$. We say that $\mathcal{M}$ satisfies $\vec{x} \botc{\vec{z}} \vec{y}$ under $X$, in symbols $\mathcal{M} \models_X \vec{x} \botc{\vec{z}} \vec{y}$, if
			\[ \forall s, s' \in X (s(\vec{z}) = s'(\vec{z}) \rightarrow \exists s'' \in X (s''(\vec{z}) =s(\vec{z}) \wedge  s''(\vec{x}) = s(\vec{x})  \wedge  s''(\vec{y}) = s'(\vec{y}))). \]
	%
	%
	%
		Let $\Sigma$ be a set of atoms and let $X$ be such that the set of variables occurring in $\Sigma$ is included in $\mathrm{dom}(X)$. We say that $\mathcal{M}$ satisfies $\Sigma$ under $X$, in symbols $\mathcal{M} \models_X \Sigma$, if $\mathcal{M}$ satisfies every atom in $\Sigma$ under $X$. 
	%
	%
		We say that $\vec{x} \botc{\vec{z}} \vec{y}$ is a logical consequence of $\Sigma$, in symbols $\Sigma \models \; \vec{x} \botc{\vec{z}} \vec{y}$, if for every $\mathcal{M}$ and $X$ such that the set of variables occurring in $\Sigma \cup \left\{ \vec{x} \botc{\vec{z}} \vec{y} \right\}$ is included in $\mathrm{dom}(X)$ we have that
			\[ \text{ if } \; \mathcal{M} \models_X \Sigma \; \text{ then } \; \mathcal{M} \models_X \vec{x} \botc{\vec{z}} \vec{y}. \]
	%
	%

		The deductive system of $\mathrm{ACIndL}$ consists of the following rules:
	\begin{enumerate}[($a_5.$)]
		\item $\vec{x} \botc{\vec{x}} \vec{y}$;
		\item If $\vec{x} \botc{\vec{z}} \vec{y}$, then $\vec{y} \botc{\vec{z}} \vec{x}$;
		\item If $\vec{x}  \vec{x}'  \botc{\vec{z}} \vec{y}  \vec{y}'$, then $\vec{x} \botc{\vec{z}} \vec{y}$;
		\item If $\vec{x} \botc{\vec{z}} \vec{y}$, then $\vec{x}  \vec{z}  \botc{\vec{z}} \vec{y}  \vec{z}$;
		\item If $\vec{x} \botc{\vec{z}} \vec{y}$ and $\vec{u} \botc{\vec{z}, \vec{x}} \vec{y}$, then $\vec{u} \botc{\vec{z}} \vec{y}$;
		\item If $\vec{y} \botc{\vec{z}} \vec{y}$ and $\vec{z}  \vec{x} \botc{\vec{y}} \vec{u}$, then $\vec{x} \botc{\vec{z}} \vec{u}$;
		\item If $\vec{x} \botc{\vec{z}} \vec{y}$ and $\vec{x}  \vec{y} \botc{\vec{z}} \vec{u}$, then $\vec{x} \botc{\vec{z}} \vec{y}  \vec{u}$;
		\item If $\vec{x} \botc{\vec{z}} \vec{y}$, then $\pi\vec{x} \botc{\tau\vec{z}} \sigma\vec{y}$ [where $\pi$, $\tau$ and $\sigma$ are permutations of $\vec{x}$, $\vec{z}$, and $\vec{y}$ respectively].

	\end{enumerate}
	The notions of deduction and provability are defined in analogy with $\mathrm{AIndL}$.


		\begin{theorem} Let $\Sigma$ be a set of atoms, then 
		\[ \Sigma \vdash \; \vec{x} \botc{\vec{z}} \vec{y} \;\;\; \Rightarrow  \;\;\; \Sigma \models \; \vec{x} \botc{\vec{z}} \vec{y}.  \] \qed

	\end{theorem}

		Parker and Parsaye-Ghomi \cite{Parker:1980:IIE:582250.582259} proved that it is not possible to find a \emph{finite} complete axiomatization for the conditional independence atoms. Furthermore, in \cite{herrmann_undec} and \cite{herrmann_undec_corrig} Hermann proved that the consequence relation between these atoms is undecidable. It is, a priori, obvious that there is {\em some} recursively enumerable axiomatization for the conditional independence atoms, because we can reduce the whole question to first order logic with extra predicates and then appeal to the Completeness Theorem of first order logic. In \cite{naumov} Naumov and Nicholls developed an explicit recursively enumerable axiomatization of them.

	\subsection{Abstract Independence Relation Atomic Independence Logic}\label{abstr_indep_rel_indep_logic}

	The system {\em Abstract Independence Relation Atomic Independence Logic} \newline ($\mathrm{AIRAIndL}(\pureindep)$) is defined as follows.
	The syntax and deductive system of this logic are the same as those of $\mathrm{AIndL}$.
	%
	%
	Let $(\mathbf{K}, \preccurlyeq)$ be an $\mathrm{AEC}$ with $\mathrm{AP}$, $\mathrm{JEP}$ and $\mathrm{ALM}$, and $\pureindep$ a pre-independence relation between (bounded) subsets of the monster model.
	%
	Let $s: \mathrm{dom}(s) \rightarrow \mathfrak{M}$ with $\vec{x}  \vec{y} \subseteq \mathrm{dom}(s) \subseteq \mathrm{Var}$. We say that $\mathfrak{M}$ satisfies $\vec{x} \boto \vec{y}$ under $s$, in symbols $\mathfrak{M} \models_s \vec{x} \boto \vec{y}$, if 
	\[ s(\vec{x}) \pureindep[\emptyset] s(\vec{y}). \]
	%
	%
	Let $\Sigma$ be a set of atoms and let $s$ be such that the set of variables occurring in $\Sigma$ is included in $\mathrm{dom}(s)$. We say that $\mathfrak{M}$ satisfies $\Sigma$ under $s$, in symbols $\mathfrak{M} \models_s \Sigma$, if $\mathfrak{M}$  satisfies every atom in $\Sigma$ under $s$. 
	%
	%
	%
	We say that $\vec{x} \boto \vec{y}$ is a logical consequence of $\Sigma$, in symbols $\Sigma \models \vec{x} \boto \vec{y}$, if for every $s$ such that the set of variables occurring in $\Sigma \cup \left\{ \vec{x} \boto \vec{y} \right\}$ is included in $\mathrm{dom}(s)$ we have that
	\[ \text{ if } \; \mathfrak{M} \models_s \Sigma \; \text{ then } \; \mathfrak{M} \models_s \; \vec{x} \boto \vec{y}. \]
	%
	%
	\smallskip

	\noindent As made clear by the notation used, the system $\mathrm{AIRAIndL}(\pureindep)$ depends on the particular pre-independence relation $\pureindep$ considered. In the next three theorems we show that, although the system is sound independently of the choice of $\pureindep$, in order to have completeness we need some further assumptions, namely federation and admissibility of an algebraic point. The generality at which the subject is developed allowed us to realize that these conditions are not only sufficient conditions for a completeness result, but also necessary. Sufficiency and necessity of these conditions are shown in Theorems \ref{completeness_indep} and \ref{characterization}, respectively.

	\begin{theorem}\label{soundness_indep} $\mathrm{AIRAIndL}(\pureindep)$ is sound.

	\end{theorem}

	\begin{proof} Let $s$ an appropriate assignment.
	\smallskip

	\noindent	($a_3.$) By Existence, $\emptyset \pureindep[\emptyset] \vec{a}$ for any $\vec{a} \in \mathfrak{M}^{< \omega}$. Thus, by Symmetry, we have $\mathfrak{M} \models_s s(\vec{x}) \pureindep[\emptyset] \emptyset$.
	\smallskip

	\noindent	($b_3.$) \[ \begin{array}{rcl}
		\mathfrak{M} \models_s \vec{x} \boto \vec{y} & \;\;\; \Longrightarrow & \;\;\; s(\vec{x}) \pureindep[\emptyset] s(\vec{y}) \\
		        											  & \;\;\; \Longrightarrow & \;\;\; s(\vec{y}) \pureindep[\emptyset] s(\vec{x}) \;\;\;\; [\text{By Symmetry}]\\
														      & \;\;\; \Longrightarrow & \;\;\; \mathfrak{M} \models_s \vec{y} \boto \vec{x}.	

	\end{array}
		 \]
	\smallskip

	\noindent	($c_3.$) \[ \begin{array}{rcl}
		\mathfrak{M} \models_s \vec{x} \boto \vec{y} \vec{z}  & \;\;\; \Longrightarrow & \;\;\; s(\vec{x}) \pureindep[\emptyset] s(\vec{y}  \vec{z}) \\
		        											  				  & \;\;\; \Longrightarrow & \;\;\; s(\vec{x}) \pureindep[\emptyset] s(\vec{y}) \;\;\;\;\;\, [\text{By Monotonicity}]\\
														      				  & \;\;\; \Longrightarrow & \;\;\; \mathfrak{M} \models_s \vec{x} \boto \vec{y}.	

	\end{array}
		 \]
	\smallskip

	\noindent	($d_3.$) \[ \begin{array}{rcl}
		 &\;\;\, \mathfrak{M} \models_s \vec{x} \boto \vec{y}  \text{ and } \mathfrak{M} \models_s \vec{x}  \vec{y} \boto \vec{z}           &  \\
		        											  &  \Downarrow & \\
		 & s(\vec{x}) \pureindep[\emptyset] s(\vec{y}) \text{ and }	s(\vec{x})s(\vec{y}) \pureindep[\emptyset] s(\vec{z})	       &  \\          
		  													  &  \Downarrow & \\
		  													  &  s(\vec{x}) \pureindep[\emptyset] s(\vec{y})s(\vec{z}) & [\text{By Exchange}]\\
		  													  &  \Downarrow & \\
		  													  &  \mathfrak{M} \models_s \vec{x} \boto \vec{y}  \vec{z}. & \\

	\end{array}
		 \]
	\smallskip

	\noindent	($e_3.$) Suppose that $\mathfrak{M} \models_s x \boto x$, then $s(x) \pureindep[\emptyset] s(x)$ and so by Anti-Reflexivity we have $s(x) \pureindep[\emptyset] s(\vec{y})$ for any $\vec{y} \in \mathrm{Var}$.
	\smallskip

	\noindent	($f_3.$) Obvious.

	\smallskip

	\noindent	($g_3.$) This is clear because $\pureindep$ is a ternary relation between {\em subsets} of $\mathfrak{M}$, rather than sequences.

	\end{proof}

		\begin{theorem}\label{completeness_indep} If $\pureindep$ is federated and admits an algebraic point, then $\mathrm{AIRAIndL}(\pureindep)$ is complete.

	\end{theorem}

	\begin{proof}

	\noindent Let $\Sigma$ be a set of atoms and suppose that $\Sigma \nvdash \vec{x} \boto \vec{y}$. Notice that if this is the case then $\vec{x} \neq \emptyset$ and $\vec{y} \neq \emptyset$. Indeed if $\vec{y} = \emptyset$ then $\Sigma \vdash \vec{x} \boto \vec{y}$ because by rule ($a_3.$) $\vdash \vec{x} \boto \emptyset$. Analogously if $\vec{x} = \emptyset$ then $\Sigma \vdash \vec{x} \boto \vec{y}$ because by rule ($a_3.$) $\vdash \vec{y} \boto \emptyset$ and so by rule ($b_3.$) $\vdash \emptyset \boto \vec{y}$. 
		Furthermore we can assume that $\vec{x} \boto \vec{y}$ is minimal, in the sense that if $\vec{x}' \subseteq \vec{x}$, $\vec{y}' \subseteq \vec{y}$ and $\vec{x}' \neq \vec{x}$ or $\vec{y}' \neq \vec{y}$, then $\Sigma \vdash \vec{x}' \boto \vec{y}'$. This is for two reasons. 
		\begin{enumerate}[i)]
			\item If $\vec{x} \boto \vec{y}$ is not minimal we can always find a minimal atom $\vec{x}^* \boto \vec{y}^*$ such that $\Sigma \nvdash \vec{x}^* \boto \vec{y}^*$, $\vec{x}^* \subseteq \vec{x}$ and $\vec{y}^* \subseteq \vec{y}$ --- just keep deleting elements of $\vec{x}$ and $\vec{y}$ until you obtain the desired property or until both $\vec{x}^*$ and $\vec{y}^*$ are singletons, in which case, due to the trivial independence rule ($a_3.$), $\vec{x}^* \boto \vec{y}^*$ is a minimal statement. 
			\item For any $\vec{x}' \subseteq \vec{x}$, $\vec{y}' \subseteq \vec{y}$ and assignment $s$ we have that if $\mathfrak{M} \not\models_s \vec{x}' \boto \vec{y}'$ then $\mathfrak{M} \not\models_s \vec{x} \boto \vec{y}$.

	\end{enumerate}
	Let indeed $\vec{x} = \vec{x}'\vec{x}''$ and $\vec{y} = \vec{y}'\vec{y}''$, then
		\[ \begin{array}{rcl}
			\mathfrak{M} \models_s \vec{x}'  \vec{x}'' \boto \vec{y}' \vec{y}''  & \;\;\; \Longrightarrow & \;\;\; s(\vec{x}') s(\vec{x}'') \pureindep[\emptyset] s(\vec{y}') s(\vec{y}'') \\
			        											  & \;\;\; \Longrightarrow & \;\;\; s(\vec{x}') \pureindep[\emptyset] s(\vec{y}')\;\;\;\;\;\;\;\;\;\;\;\;\;\;\;\;\; [\text{By Monotonicity}]\\

															      & \;\;\; \Longrightarrow & \;\;\; \mathfrak{M} \models_s \vec{x}' \boto \vec{y}'.	

	\end{array}
			 \]
	\smallskip 

	\noindent	Let $V = \left\{ v \in \mathrm{Var} \; | \; \Sigma \vdash v \boto v \right\}$, $W = \mathrm{Var} - V$, $\vec{x} \cap W = \vec{x}'$ and $\vec{y} \cap W = \vec{y}'$.
	\smallskip

	\noindent \begin{claim}\label{claim_for_proof} If $\Sigma \vdash \vec{x}' \boto \vec{y}'$, then $\Sigma \vdash \vec{x} \boto \vec{y}$.

	\end{claim}

		\begin{claimproof} Let $\vec{x} - \vec{x}' = (x_{s_0}, ..., x_{s_{b-1}})$ and $\vec{y} - \vec{y}' = (y_{g_0}, ..., y_{g_{c-1}})$, then by rules ($e_3.$), ($b_3.$) and ($d_3.$) we have that

			\[ \begin{array}{rcl}
			  \Sigma \vdash \vec{x}' \boto \vec{y}'  & \text{ and } & \Sigma \vdash \vec{x}'  \vec{y}' \boto y_{g_0}            \\
			        											  &  \Downarrow & \\
			 &	\Sigma \vdash \vec{x}' \boto \vec{y}'  y_{g_0}         &  \\
			  													  &  \vdots & \\
			  \Sigma \vdash \vec{x}' \boto \vec{y}' y_{g_0} \cdots y_{g_{c-2}}  & \text{ and } & \Sigma \vdash \vec{x}'  \vec{y}'  y_{g_0} \cdots y_{g_{c-2}} \boto y_{g_{c-1}}      \\
															  	&  \Downarrow & \\
			&	\Sigma \vdash \vec{x}' \boto \vec{y}'  y_{g_0} \cdots y_{g_{c-1}}        &  
		\end{array}
			 \]
	and hence by rules ($f_3.$) and ($b_3.$) we have that $\Sigma \vdash \vec{y} \boto \vec{x}'$. Thus 

		\[ \begin{array}{rcl}
		  \Sigma \vdash \vec{y} \boto \vec{x}'  & \text{ and } & \Sigma \vdash \vec{y}  \vec{x}' \boto x_{s_0}            \\
		        											  &  \Downarrow & \\
		 &	\Sigma \vdash \vec{y} \boto \vec{x}'  x_{s_0}         &  \\
		  													  &  \vdots & \\
		  \Sigma \vdash \vec{y} \boto \vec{x}'  x_{s_0} \cdots x_{s_{b-2}} & \text{ and } & \Sigma \vdash \vec{y}  \vec{x}'  x_{s_0} \cdots x_{s_{b-2}} \boto x_{s_{b-1}}      \\
														  	&  \Downarrow & \\
		&	\Sigma \vdash \vec{y} \boto \vec{x}'   x_{s_0} \cdots x_{s_{b-1}}      &  
	\end{array}
		 \]
	and hence by rule ($f_3.$) and ($b_3.$) we have that $\Sigma \vdash \vec{x} \boto \vec{y}$.

	\end{claimproof}

	\smallskip 

	\noindent The claim above shows that if $\vec{x} \boto \vec{y}$ is minimal, then for every $z \in \vec{x}\vec{y}$ we have that $\Sigma \nvdash z \boto z$. Furthermore, because of rule ($g_3.$) we can assume that $\vec{x}$ and $\vec{y}$ are injective. This will be relevant in the following. We now make a case distinction. 
	\smallskip 

	\noindent \textbf{Case 1.} There exists $z \in \vec{x} \cap \vec{y}$. Notice that by assumption there exist an algebraic point $e \in \mathfrak{M}$ and a federated sequence $(a_0) \in \mathfrak{M}^1$. Let $s$ be the following assignment:	
	\begin{enumerate}[i)]
			\item $s(v) = e$ for every $v \in \mathrm{Var} - z$,
			\item $s(z) = d$.
	\end{enumerate}
	where $d$ is obtained using the defining condition of federation of the independent sequence $(a_0)$. 

	\smallskip

	\noindent Obviously $\mathfrak{M} \not\models_s \vec{x} \boto \vec{y}$, in fact $\mathfrak{M} \not\models_s z \boto z$. This is because if $d \pureindep[\emptyset] d$, then $d \pureindep[\emptyset] a_0$, contrary to the choice of $d$. Furthermore, for every $\vec{v} \boto \vec{w} \in \Sigma$, we have that $\mathfrak{M} \models_s \vec{v} \boto \vec{w}$. Let indeed $\vec{v} \boto \vec{w} \in \Sigma$, then $z \notin \vec{v} \cap \vec{w}$, because otherwise, by rule ($c_3.$), we would have that $\Sigma \vdash z \boto z$, contrary to the minimality of $\vec{x} \boto \vec{y}$. Hence $\mathfrak{M} \models_s \vec{v} \boto \vec{w}$, because by the choice of $e$ we have that $e \pureindep[\emptyset] ed$.

	\smallskip

	\noindent	\textbf{Case 2.} $\vec{x} \cap \vec{y} = \emptyset$.

	\noindent Let $\vec{x} = (x_0, ..., x_{n-1})$, $\vec{y} = (y_0, ..., y_{m-1})$ and $k = (n-1) + m$. Let then $(w_i \; | \; i < k)$ be an injective enumeration of $\vec{x}  \vec{y} - x_0$ with $w_i = x_{i+1}$ for $i \leq n-2$ and $w_{i+{(n-1)}} = y_{i}$ for $i \leq m-1$. Notice that by assumption there exist a constant point $e \in \mathfrak{M}$ and an federated sequence $(a_i \; | \; i < k) \in \mathfrak{M}^{k}$.
		Let then $s$ be the following assignment:
		\begin{enumerate}[i)]
			\item $s(v) = e$ for every $v \in \mathrm{Var} - \vec{x}  \vec{y}$,
			\item $s(w_i) = a_i $ for every $i < k$,
			\item $s(x_{0}) = d$,
	\end{enumerate}
	where $d$ is obtained using the defining condition of federation of the independent sequence $(a_i \; | \; i < k)$.

	\smallskip

	\noindent	We claim that $\mathfrak{M} \not\models_s \vec{x} \boto \vec{y}$. By the choice of $d$, we have that $d \not\!\pureindep[\emptyset] a_0 \cdots a_{k-1}$. Suppose that $d a_0 \cdots a_{n-2} \pureindep[\emptyset] a_{n-1} \cdots a_{k-1}$. Again by the choice of $d$, we have $d \pureindep[\emptyset] a_0 \cdots a_{n-2}$, so by Exchange we have $d \pureindep[\emptyset] a_0 \cdots a_{k-1}$, a contradiction. Thus, $d a_0 \cdots a_{n-2} \not\!\pureindep[\emptyset] a_{n-1} \cdots a_{k-1}$ and hence $s(\vec{x}) \not\!\pureindep[\emptyset] s(\vec{y})$.
	\smallskip

	\noindent	Let now $\vec{v} \boto \vec{w} \in \Sigma$, we want to show that $\mathfrak{M} \models_s \vec{v} \boto \vec{w}$. 
	Let $\vec{v} \cap \vec{x}  \vec{y} = \vec{v}'$ and $\vec{w} \cap \vec{x}  \vec{y} = \vec{w}'$.
		Notice that 
		\[ s(\vec{v}) \pureindep[\emptyset] s(\vec{w}) \text{ if and only if } s(\vec{v}') \pureindep[\emptyset] s(\vec{w}'). \]
	Left to right holds in general. As for the other direction, suppose that $s(\vec{v}') \pureindep[\emptyset] s(\vec{w}')$. If $u \in \vec{v}  \vec{w} - \vec{v}'  \vec{w}'$, then $s(u) = e$. Thus

		 \[ \begin{array}{rcl}
		 &\;\;\; s(\vec{v}') \pureindep[\emptyset] s(\vec{w}')  \text{ and } s(\vec{v}')  s(\vec{w}') \pureindep[\emptyset] e         & [\text{By Anti-Reflexivity}]  \\
		        											  &  \Downarrow & \\
		 &\;\;\;\;\, s(\vec{v}')  \pureindep[\emptyset] s(\vec{w}')  e          &  \\
		  													  &  \Downarrow & \\
		  													  &  \;\,\,s(\vec{v}')  \pureindep[\emptyset] s(\vec{w}). &
	\end{array}
		 \]
	So
		\[ \begin{array}{rcl}
	 	&\;\;\; s(\vec{w}) \pureindep[\emptyset] s(\vec{v}')  \text{ and } s(\vec{w})  s(\vec{v}') \pureindep[\emptyset] e          & 			   [\text{By Anti-Reflexivity}]  \\
	        											  &  \Downarrow & \\
	 	&\;\;\;\; s(\vec{w})  \pureindep[\emptyset] s(\vec{v}')  e          &  \\
	  													  &  \Downarrow & \\
	  													  &  \;\,s(\vec{w})  \pureindep[\emptyset] s(\vec{v}). &
	\end{array}
	 \]

	\smallskip

	\noindent Notice that $\vec{v}' \cap \vec{w}' = \emptyset$. Indeed, suppose that there exists $z \in \vec{v}' \cap \vec{w}'$, then, by rule ($c_3.$), we have that $\Sigma \vdash z \boto z$, contrary to the minimality of $\vec{x} \boto \vec{y}$. We make another case distinction.
	\smallskip

	\noindent	\textbf{Subcase 1.} $x_{0} \notin \vec{v}' \vec{w}'$. 
		 As noticed, $\vec{v}' \cap \vec{w}' = \emptyset$, and so, by properties of our assignment $s(\vec{v}') \cap s(\vec{w}') = \emptyset$. Thus, by Lemma~\ref{lemma_indep_seq}, it follows that $s(\vec{v}') \pureindep[\emptyset] s(\vec{w}')$.
	\smallskip

	\noindent	\textbf{Subcase 2.} $x_{0} \in \vec{v}' \vec{w}'$. 
	\smallskip

	\noindent	\textbf{Subcase 2.1.} $(\vec{x}  \vec{y}) - (\vec{v}'  \vec{w}') \neq \emptyset$. 
			Let $\vec{a} = s(\vec{v}') - d $ and $\vec{b} = s(\vec{w}') - d$. By assumption we have that $(\vec{x} - x_0) \cup \vec{y} \not\subseteq \vec{v}'  \vec{w}'$ and so $s((\vec{x} - x_0) \cup \vec{y}) \not\subseteq \vec{a}  \vec{b}$. Thus, by the choice of $d$, we have that $\vec{a} \vec{b} \pureindep[\emptyset] d$. Suppose now that $x_{0} \in \vec{w}'$, the other case is symmetrical. By properties of our assignment $\vec{a} \cap \vec{b} = \emptyset$, hence by Lemma~\ref{lemma_indep_seq}, we have that $\vec{a} \pureindep[\emptyset] \vec{b}$. Thus, by Exchange, $\vec{a} \pureindep[\emptyset] \vec{b} d$. Hence, permuting the elements in $\vec{b} d$, we conclude that $s(\vec{v}) \pureindep[\emptyset] s(\vec{w})$.
	\smallskip

	\noindent	\textbf{Subcase 2.2.} $\vec{x}  \vec{y} \subseteq \vec{v}'  \vec{w}'$. 
			This case is not possible. By rule ($f_3.$) and ($c_3.$) we can assume that $\vec{v} = \vec{v}'  \vec{u}$ and $\vec{w} = \vec{w}'  \vec{u}'$ with $\vec{u}  \vec{u}' \subseteq \mathrm{Var} - \vec{v}'  \vec{w}'$. Furthermore because $\vec{x}  \vec{y} \subseteq \vec{v}'  \vec{w}'$ again by rule ($f_3.$) we can assume that $\vec{v}' = \vec{x}'  \vec{y}'  \vec{z}'$ and $\vec{w}' = \vec{x}''  \vec{y}''  \vec{z}''$ with $\vec{x}'  \vec{x}'' = \vec{x}$, $\vec{y}'  \vec{y}'' = \vec{y}$ and $\vec{z}' \vec{z}'' \subseteq \mathrm{Var} - \vec{x}  \vec{y}$. Hence $\vec{v} = \vec{x}'  \vec{y}'  \vec{z}'  \vec{u}$ and $\vec{w} = \vec{x}''  \vec{y}''  \vec{z}''  \vec{u}'$.
			By hypothesis we have that $\vec{v} \boto \vec{w} \in \Sigma$ so by rules ($c_3.$) and ($b_3.$) we can conclude that $\Sigma \vdash \vec{x}'  \vec{y}' \boto \vec{x}''  \vec{y}''$.
			If $\vec{x}' = \vec{x}$ and $\vec{y}'' = \vec{y}$, then $\Sigma \vdash \vec{x} \boto \vec{y}$ because as we noticed $\vec{v}' \cap \vec{w}' = \emptyset$, a contradiction. Analogously if $\vec{x}'' = \vec{x}$ and $\vec{y}' = \vec{y}$, then $\Sigma \vdash \vec{y} \boto \vec{x}$. Thus by rule ($b_3.$) $\Sigma \vdash \vec{x} \boto \vec{y}$, a contradiction. There are then four cases:
			\begin{enumerate}[i)]
				\item $\vec{x}' \neq \vec{x}$ and $\vec{x}'' \neq \vec{x}$;
				\item $\vec{y}' \neq \vec{y}$ and $\vec{x}'' \neq \vec{x}$;
				\item $\vec{y}' \neq \vec{y}$ and $\vec{y}'' \neq \vec{y}$;
				\item $\vec{x}' \neq \vec{x}$ and $\vec{y}'' \neq \vec{y}$.

		\end{enumerate}
			Suppose that either i) or ii) holds. If this is the case, then $\Sigma \vdash \vec{x}' \boto \vec{y}'$ because by hypothesis $\vec{x} \boto \vec{y}$ is minimal. So $\Sigma \vdash \vec{x}' \boto \vec{y}'  \vec{x}''  \vec{y}''$, because by rule ($d_3.$) 
			\[ \Sigma \vdash \vec{x}' \boto \vec{y}' \text{ and } \Sigma \vdash \vec{x}'  \vec{y}' \boto \vec{x}''  \vec{y}''\ \Rightarrow \Sigma \vdash \vec{x}' \boto \vec{y}'  \vec{x}''  \vec{y}''.\]
			Hence by rule ($e_3.$) $\Sigma \vdash \vec{x}' \boto \vec{x}''  \vec{y}$ and then by rule ($b_3.$)  $\Sigma \vdash \vec{x}''  \vec{y} \boto \vec{x}'$. So by rule ($e_3.$) $\Sigma \vdash \vec{y}  \vec{x}'' \boto \vec{x}'$.
			We are under the assumption that $\vec{x}'' \neq \vec{x}$ thus again by minimality of $\vec{x} \boto \vec{y}$ we have that $\Sigma \vdash \vec{x}'' \boto \vec{y}$ and so by rule ($b_3.$) we conclude that $\Sigma \vdash \vec{y} \boto \vec{x}''$. Hence $\Sigma \vdash \vec{y} \boto \vec{x}''  \vec{x}'$, because by rule ($d_3.$) 
			\[ \Sigma \vdash \vec{y} \boto \vec{x}'' \text{ and } \Sigma \vdash \vec{y}  \vec{x}'' \boto \vec{x}' \Rightarrow \Sigma \vdash \vec{y} \boto \vec{x}''  \vec{x}'.\]
			Then finally by rules ($e_3.$) and ($b_3.$) we can conclude that $\Sigma \vdash \vec{x} \boto \vec{y}$, a contradiction.
			The case in which either iii) or iv) holds is symmetrical.
\smallskip

\noindent This concludes the proof of the theorem.

	\end{proof}
	
	\begin{theorem}\label{characterization} If $\mathrm{AIRAIndL}(\pureindep)$ is complete, then $\pureindep$ is federated and admits an algebraic point.	
		
\end{theorem} 

	\begin{proof} Suppose that $\pureindep$ is not federated and let $n \in \omega^*$ witness this. Let $\Sigma$ be the following set of atoms
		\[ \left\{ x_{0} \cdots x_{i-1} \boto x_{i} \, | \, i < n \right\} \cup \left\{ x_{0} \cdots x_{i-1} x_{i+1} \cdots x_{n-1} \boto y \, | \, i < n \right\},\]
where $\vec{x} = (x_{0}, ..., x_{i-1})$ and $\vec{x} \cap y = \emptyset$.
	Then we have the following validity
			\[\tag{$*$} \Sigma \models \vec{x} \boto y, \]	
but clearly ($*$) is not deducible in our deductive system. Indeed, the theory $\mathrm{VS}^{\mathrm{inf}}_{\mathbb{Q}}$ of non-trivial vector spaces over the field $\mathbb{Q}$ of rational numbers is a counterexample. Let $s: \mathrm{Var} \rightarrow \mathbb{Q}^{n} \preccurlyeq \mathcal{M}$ be the following assignment:
	\[s(v) = \begin{cases} e_i \;\;\;\;\; \text{ if } v = x_{i} \\
						   1 \;\;\;\;\;\; \text{ if } v = y  \\
						   0  \;\;\;\;\;\; \text{ otherwise, }
\end{cases}\]
where, for $j < n$, $e_i(j) = 1$ if $i = j$ and $0$ otherwise. Then clearly 
\[ \mathfrak{M} \models_s \Sigma \text{ but } \mathfrak{M} \not \models_s \vec{x} \boto y. \]
\smallskip
 
\noindent	Suppose that $\pureindep$ does not admit an algebraic point. Then we have we have the following validity
		\[\tag{$**$} x \boto x \models y \boto z, \]
where $y, z \neq x$. But clearly ($**$) is not deducible in our deductive system. Indeed, again $\mathrm{VS}^{\mathrm{inf}}_{\mathbb{Q}}$ is a counterexample. Let $s: \mathrm{Var} \rightarrow \mathbb{Q} \preccurlyeq \mathcal{M}$ be the following assignment:
		\[s(v) = \begin{cases} 0 \;\;\;\;\; \text{ if } v = x \\
							   1 \;\;\;\;\; \text{ if } v = y   \\
							   2  \;\;\;\;\; \text{ otherwise, }
	\end{cases}\]
Then clearly 
	\[ \mathfrak{M} \models_s x \boto x \text{ but } \mathfrak{M} \not \models_s y \boto z. \]

\end{proof}

 As already noticed, the semantics of $\mathrm{AIRAIndL}(\pureindep)$ is parametrized by an $\mathrm{AEC}$ and a fixed pre-independence relation. It is possible to formulate a ``cousin system'' of $\mathrm{AIRAIndL}(\pureindep)$ where the dependency from a particular pre-independence relation is dropped. We do this. Let $(\mathbf{K}, \preccurlyeq)$ be an $\mathrm{AEC}$ with $\mathrm{AP}$, $\mathrm{JEP}$ and $\mathrm{ALM}$, and $\mathfrak{M}$ its monster model. For a pre-independence relation $\pureindep$ on $\mathfrak{M}$, we denote by $\models_{\pureindep}$ the semantical relation of $\mathrm{AIRAIndL}(\pureindep)$. Let $\Sigma$ be a set of atoms, we say that $\vec{x} \boto \vec{y}$ is a logical consequence of $\Sigma$, in symbols $\Sigma \models^* \vec{x} \boto \vec{y}$, if for every pre-independence relation $\pureindep$ on $\mathfrak{M}$ we have that 
\[ \Sigma \models_{\pureindep} \vec{x} \boto \vec{y}. \]

\begin{theorem} Let $\Sigma$ a set of independence atoms. The following are equivalent.
	\begin{enumerate}[(1)]
		\item For some $\pureindep$ which is federated and admits an algebraic point, $\Sigma \models_{\pureindep} \vec{x} \boto \vec{y}$.
		\item For {\em any} $\pureindep$ which is federated and admits an algebraic point, $\Sigma \models_{\pureindep} \vec{x} \boto \vec{y}$.
		\item $\Sigma \models^* \vec{x} \boto \vec{y}$.
		\item $\Sigma \vdash \vec{x} \boto \vec{y}$.
	\end{enumerate}	
\end{theorem}

\begin{proof} (4) implies (3) by soundness (Theorem~\ref{soundness_indep}). (3) implies (2) is trivial, as is (2) implies (1). Finally, (1) implies (4) by Theorem~\ref{completeness_indep}.
	
\end{proof}

We now introduce a conditional version of ($\mathrm{AIRAIndL}(\pureindep)$). The system {\em Abstract Independence Relation Atomic Conditional Independence Logic} ($\mathrm{AIRACIndL}(\pureindep)$) is defined as follows.
	The syntax and deductive system of this logic are the same as those of $\mathrm{ACIndL}$. 
%
%
Let $(\mathbf{K}, \preccurlyeq)$ be an $\mathrm{AEC}$ with $\mathrm{AP}$, $\mathrm{JEP}$ and $\mathrm{ALM}$, and $\pureindep$ a pre-independence relation between (bounded) subsets of the monster model $\mathfrak{M}$.
%
	Let $s: \mathrm{dom}(s) \rightarrow \mathfrak{M}$ with $\vec{x}  \vec{y}  \vec{z} \subseteq \mathrm{dom}(s) \subseteq \mathrm{Var}$. We say that $\mathfrak{M}$ satisfies $\vec{x} \botc{\vec{z}} \vec{y}$ under $s$, in symbols $\mathfrak{M} \models_s \vec{x} \botc{\vec{z}} \vec{y}$, if 
	\[ s(\vec{x}) \pureindep[s(\vec{z})] s(\vec{y}). \]
%
%
	Let $\Sigma$ be a set of atoms and let $s$ be such that the set of variables occurring in $\Sigma$ is included in $\mathrm{dom}(s)$. We say that $\mathfrak{M}$ satisfies $\Sigma$ under $s$, in symbols $\mathfrak{M} \models_s \Sigma$, if $\mathfrak{M}$  satisfies every atom in $\Sigma$ under $s$. 
%
%
%
	We say that $\vec{x} \botc{\vec{z}} \vec{y}$ is a logical consequence of $\Sigma$, in symbols $\Sigma \models \vec{x} \botc{\vec{z}} \vec{y}$, if for every $s$ such that the set of variables occurring in $\Sigma \cup \left\{ \vec{x} \botc{\vec{z}} \vec{y} \right\}$ is included in $\mathrm{dom}(s)$ we have that
\[ \text{ if } \; \mathfrak{M} \models_s \Sigma \; \text{ then } \; \mathfrak{M} \models_s \; \vec{x} \botc{\vec{z}} \vec{y}. \]
%
%
%

\smallskip

\noindent We now show that the system $\mathrm{AIRACIndL}(\pureindep)$ is sound. The proof of this theorem is completely standard, but presented ``for the benefit of the reader".

	\begin{theorem} $\mathrm{AIRACIndL}(\pureindep)$ is sound.

\end{theorem}

	\begin{proof} Let $s$ an appropriate assignment.
\smallskip

\noindent	($a_5.$) By Existence for any $\vec{a}, \vec{b} \in \mathfrak{M}^{< \omega}$, $\vec{a} \pureindep[\vec{a}] \vec{b}$.
\smallskip
		
\noindent	($b_5.$) \[ \begin{array}{rcl}
		\mathfrak{M} \models_s \vec{x} \botc{\vec{z}} \vec{y} & \;\;\; \Longrightarrow & \;\;\; s(\vec{x}) \pureindep[s(\vec{z})] s(\vec{y}) \\
		        											  & \;\;\; \Longrightarrow & \;\;\; s(\vec{y}) \pureindep[s(\vec{z})] s(\vec{x}) \;\;\;\;\;\; [\text{By Symmetry}]\\
														      & \;\;\; \Longrightarrow & \;\;\; \mathfrak{M} \models_s \vec{y} \botc{\vec{z}} \vec{x}.	
		
	\end{array}
		 \]
\smallskip
	
\noindent	($c_5.$) \[ \begin{array}{rcl}
		\mathfrak{M} \models_s \vec{x}  \vec{x}' \botc{\vec{z}} \vec{y}  \vec{y}'  & \;\;\; \Longrightarrow & \;\;\; s(\vec{x}) s(\vec{x}') \pureindep[s(\vec{z})] s(\vec{y}) s(\vec{y}') \\
																			  & \;\;\; \Longrightarrow & \;\;\; s(\vec{x}) \pureindep[s(\vec{z})] s(\vec{y})
\;\;\; [\text{By Monotonicity}] \\
														      				  & \;\;\; \Longrightarrow & \;\;\; \mathfrak{M} \models_s \vec{x} \botc{\vec{z}} \vec{y}.	

	\end{array}
		 \]
\smallskip

\noindent	($d_5.$) Suppose that $\mathfrak{M} \models_s \vec{x} \botc{\vec{z}} \vec{y}$, then $s(\vec{x}) \pureindep[s(\vec{z})] s(\vec{y})$ and so by Normality $s(\vec{x})  s(\vec{z}) \pureindep[s(\vec{z})] s(\vec{y})$. Now, by Symmetry $s(\vec{y}) \pureindep[s(\vec{z})] s(\vec{x})  s(\vec{z})$, hence again by Normality $s(\vec{y})  s(\vec{z}) \pureindep[s(\vec{z})] s(\vec{x})  s(\vec{z})$, and thus, by Symmetry, $s(\vec{x})  s(\vec{z}) \pureindep[s(\vec{z})] s(\vec{y})  s(\vec{z})$.
\smallskip
		
\noindent	($e_5.$) \[ \begin{array}{rcl}
		\mathfrak{M} \models_s \vec{x} \botc{\vec{z}} \vec{y}	 &  & \mathfrak{M} \models_s \vec{u} \botc{\vec{z}, \vec{x}} \vec{y}  \\
		     	\Downarrow	 \;\;\;\;\;\;\;\;\;		                &  & 				\;\;\;\;\;\;\;\;\;\;\;\;\; \Downarrow \\
		s(\vec{x}) \pureindep[s(\vec{z})] s(\vec{y}) &   & s(\vec{u}) \pureindep[s(\vec{z}), s(\vec{x})] s(\vec{y})    \\  
															  &  \Downarrow & \\
							& s(\vec{x}) s(\vec{u}) \pureindep[s(\vec{z})] s(\vec{y})	& \;\;\;\;\;\, [\text{By Transitivity}]  \\
															  &  \Downarrow & \\
						    & s(\vec{u}) s(\vec{x}) \pureindep[s(\vec{z})] s(\vec{y})	& \;\;\;\;\; \\     
		  													  &  \Downarrow & \\
							& s(\vec{u}) \pureindep[s(\vec{z})] s(\vec{y}) & \;\;\;\;\;\,  [\text{By Monotonicity}] \\
		  													  &  \Downarrow & \\
		  					&  \mathfrak{M} \models_s \vec{u} \botc{\vec{z}} \vec{y} & \\
		
\end{array} \]
\smallskip
		
\noindent	($f_5.$) \[ \begin{array}{rcl}
	\;\;\;\;\;\;\;\;\;\;\;\;	\mathfrak{M} \models \vec{y} \botc{\vec{z}} \vec{y} &      & \mathfrak{M} \models_s \vec{z}  \vec{x} \botc{\vec{y}} \vec{u}   \\
 				\Downarrow	 \;\;\;\;\;\;\;\;\;		                &  & 				\;\;\;\;\;\;\;\;\;\;\;\;\; \Downarrow \\
		s(\vec{y}) \pureindep[s(\vec{z})] s(\vec{y}) &      & s(\vec{z}) s(\vec{x}) \pureindep[s(\vec{y})] s(\vec{u})  \\  
     			\Downarrow	 \;\;\;\;\;\;\;\;\;		                &  & 				\;\;\;\;\;\;\;\;\;\;\;\;\; \Downarrow \\
		s(\vec{y}) \pureindep[s(\vec{z})] s(\vec{y}) & 	& s(\vec{x}) \pureindep[s(\vec{y}), s(\vec{z})] s(\vec{u})  \;\;\;\;\;\, [\text{By Transitivity}]  \\
     			\Downarrow	 \;\;\;\;\;\;\;\;\;		                &  & 				\;\;\;\;\;\;\;\;\;\;\;\;\; \Downarrow \\
		s(\vec{y}) \pureindep[s(\vec{z})] s(\vec{u}) & 	& s(\vec{x}) \pureindep[s(\vec{y}), s(\vec{z})] s(\vec{u})  \;\;\;\;\;\,  \\
     			\Downarrow	 \;\;\;\;\;\;\;\;\;		                &  & 				\;\;\;\;\;\;\;\;\;\;\;\;\; \Downarrow \\
		s(\vec{y}) \pureindep[s(\vec{z})] s(\vec{u}) & 	& s(\vec{x}) \pureindep[s(\vec{z}), s(\vec{y})] s(\vec{u}) \;\;\;\;\; \\     
		  													  &  \Downarrow & \\
   	     					&   s(\vec{x}) \pureindep[s(\vec{z})] s(\vec{u}) & \;\;\;\;\; [\text{By what we showed in (e$_5$.)}]\\
		  													  &  \Downarrow & \\
		  			 						&  \mathfrak{M} \models_s \vec{x} \botc{\vec{z}} \vec{u} & \\
		
	\end{array}
		 \]
\smallskip

\noindent	($g_5.$) \[ \begin{array}{rcl}
		 &\;\;\, \mathfrak{M} \models_s \vec{x} \botc{\vec{z}} \vec{y}  \text{ and } \mathfrak{M} \models_s \vec{x}  \vec{y} \botc{\vec{z}} \vec{u}           &  \\
		        											  &  \Downarrow & \\
		 & s(\vec{x}) \pureindep[s(\vec{z})] s(\vec{y}) \text{ and }	s(\vec{x})s(\vec{y}) \pureindep[s(\vec{z})] s(\vec{u})	       &  \\          
		  													  &  \Downarrow & \\
		  													  &  s(\vec{x}) \pureindep[s(\vec{z})] s(\vec{y})s(\vec{u}) & \;\;\;\;\; [\text{By Exchange}]\\
		  													  &  \Downarrow & \\
		  													  &  \mathfrak{M} \models_s \vec{x} \botc{\vec{z}} \vec{y}  \vec{u} & \\
		
	\end{array}
		 \]
\smallskip
	
\noindent	($h_5.$) Obvious.
		
\end{proof}

	The system $\mathrm{AIRACIndL}(\pureindep)$ is not in general complete. In fact, in function of the validities that the pre-independence relation determines, one may need to add axioms to the deductive system. For example, several forms of triviality may occur, and our deductive system does not account for them. In some cases the axiomatization may even not be finite or recursive.  As in the case of statistics and database theory, the question of completeness for the conditional independence atom is a non-trivial one.

\section{Conclusion}

	We generalized the results of \cite{PaoliniVaananen1} to the framework of abstract independence relations for an arbitrary $\mathrm{AEC}$, which subsumes most of the cases of independence of interest in model theory. We introduced the notion of federated pre-independence relation and studied important examples of this form of independence. We showed that any $\omega$-homogenous non-trivial pregeometry is federated (modulo a finite localization), and used this result to deduce that in any first-order stable theory that admits non-trivial regular types forking independence is federated (over some set of parameters). Finally, we characterized federation and existence of an algebraic point as the model-theoretic analog of the form of independence studied in independence logic and statistics, proving that the implication problem for a pre-independence relation $\pureindep$ is solvable with respect to the deductive system that axiomatizes independence in team semantics if and only if $\pureindep$ is federated and admits an algebraic point.

\end{document}